\documentclass[10pt,a4paper]{article}
\usepackage[english]{babel}
\usepackage[latin1]{inputenc}
\usepackage[all]{xy}
\usepackage{amsmath}
\usepackage{amssymb}
\usepackage{amsthm}
\usepackage{graphicx}
\usepackage{mathrsfs}
\usepackage{amsmath,amsfonts,amssymb,amsthm}
\newtheorem{thm}{Theorem}
\newtheorem{de}{Definition}
\newtheorem{rem}{Remark}

\newtheorem{conj}{Conjecture}
\newtheorem{for}{Formulation}
\newtheorem{st}{Statement}
\newtheorem{lem}{Lemma}
\newtheorem{prop}{Proposition}

\newtheorem{cor}{Corollary}
\newcommand{\ov}{\overline}

\title{Some examples toward a Manin-Mumford conjecture for abelian uniformizable $T-$modules}
\author{Luca Demangos\\Former address: Laboratoire Paul Painlevé, USTL,\\Batiment M2 Cité Scientifique,\\59655 Villeneuve d'Ascq Cédex\\Current address: Instituto de Matem\'{a}ticas -- Unidad Cuernavaca,\\Universidad
Nacional Autonoma de M\'{e}xico,\\Av. Universidad S/N, C.P. 62210\\Cuernavaca, Morelos, M\'{E}XICO\\e-mail: l.demangos@gmail.com}

\begin{document}
\maketitle

The aim of this work is to present a possible adaptation of the Manin-Mumford conjecture to the $T-$modules, a mathematical object which has been introduced in the 1980's by G. Anderson as the natural analogue of the abelian varieties in the context of modules over rings which are contained in positive characteristic function fields. We propose then a generalisation of such an adapted conjecture to a modified general version of Mordell-Lang conjecture for $T-$modules which might correct the one proposed for the first time by L. Denis in \cite{Denis'} but no longer compatible with the present results.\\\\We will remind in our first preliminary section the formulation of the Mordell-Lang and the Manin-Mumford conjectures and the definition of $T-$modules and sub-$T-$modules, listing nothing more than the basic definitions and properties which are strictly essential to us in order to state and prove our theorems. All the detailed informations the reader may be interested in might be found in \cite{Goss}, chapter 5. We then trace a brief history of the recent studies, mainly due to the work of L. Denis, D. Ghioca and T. Scanlon, about an adaptation of such conjectures to a very special case of $T-$module: the power of a Drinfeld module. In the second section we will present some examples which prove that a naif adaptation of the Manin-Mumford conjecture to the $T-$modules in general is no longer true, even extending the notion of sub-$T-$module in a way which encode the deeper structure of the involved rings. In one of these examples we will take the case of a product of different Drinfeld modules, showing that such an intuitive adaptation of Manin-Mumford conjecture is no longer true even in this appearently simple case of study. We will then propose a reasonably corrected formulation, whose proof is the final aim of our research. 
\section{Preliminaries}
G. Faltings proved the Mordell-Lang conjecture (see \cite{F}) in the following version.
\begin{thm}
Let $\mathcal{A}$ be an abelian variety defined over a number field. Let $X$ be a closed subvariety of $\mathcal{A}$ and $\Gamma\subset \mathcal{A}$ a finitely generated subgroup of the group of $\mathbb{C}-$points on $\mathcal{A}$. Then $X\cap \Gamma$ is a finite union of cosets of subgroups of $\Gamma$.
\end{thm}
Such a result contains clear similarities with the Manin-Mumford conjecture. The classic version of this conjecture\footnote{Historically, the first version of the Manin-Mumford conjecture has been proved by M. Laurent over $\mathbb{G}_{m}^{n}$, see \cite{Lau2}}, which treats the situation of an elliptic curve embedded in its Jacobian variety has been proved (see \cite{R}) and later generalized (see \cite{R2}) to the following statement.\\\\
We call \textbf{torsion subvariety} of an abelian variety $\mathcal{A}$ the translate of some abelian subvariety of $\mathcal{A}$ by a torsion point.
\begin{thm}
Let $X$ be an algebraic subvariety of an abelian variety $\mathcal{A}$ defined over a number field. If $X$ contains a Zariski-dense set of torsion points, then $X$ is a torsion subvariety of $\mathcal{A}$. 
\end{thm}
An even stronger evidence of the link existing between these two results is given by the following extended formulation of G. Faltings' result, proved by M. Hindry and G. Faltings (see \cite{H}):
\begin{thm}
Let $X$ be a closed subvariety of an abelian variety $\mathcal{A}$ defined over a number field. Let $\Gamma$ be a finitely generated subgroup of $\mathcal{A}(\ov{\mathbb{Q}})$ and let:\[\ov{\Gamma}:=\{\ov{x}\in \mathcal{A}(\ov{\mathbb{Q}}),\exists m>0,[m]\ov{x}\in \Gamma\}.\]Therefore, $X(\ov{\mathbb{Q}})\cap \ov{\Gamma}$ is a finite union of cosets of subgroups of $\ov{\Gamma}$.
\end{thm}
In fact, the Manin-Mumford conjecture (proved before such a result was known) follows if one chooses $\Gamma=\ov{0}$ in Theorem 3. On the other hand, the study of the Manin-Mumford conjecture still produces relevant ideas, which might be applied to different aspects of this kind of problems, as we hope to do.\\\\
The weak version of the M. Raynaud's theorem proved in \cite{PZ} using new ideas involving Diophantine Geometry is for example a consequence of Theorem 2:
\begin{thm}
Let $X$ be an algebraic subvariety of an abelian variety $\mathcal{A}$ defined over a number field. If $X$ does not contain any torsion subvariety of $\mathcal{A}$ of dimension $>0$, then $X$ contains at most finitely many torsion points.
\end{thm}
The techniques involved in proving Theorem 4 appear to particularly adapt to the study of a specific class of objects, called $T-$modules, which are affine algebraic varieties over function fields provided with a module structure, which we briefly introduce now. Most of definitions and properties we list below are taken from \cite{Goss}, chapter 5. We send the reader to such a reference for more details.\\\\
We call $A:=\mathbb{F}_{q}[T]$ the ring of polynomials with coefficient in the finite field $\mathbb{F}_{q}$, where $q$ is a power of the prime number $p$, $k:=Frac(A)$ and $\mathcal{C}$ is the completion of an algebraic closure of the completion of $k$ with respect to the place at infinity. If $K$ is a field and $n,m$ are two positive integers, the notation $K^{n,m}$ will indicate the ring of matrices with entries in $K$, having $n$ lines and $m$ columns. We use $\tau$ to indicate the Frobenius automorphism in the following form:\[\tau: z\mapsto z^{q},\texttt{   }\forall z\in \mathcal{C}.\]
\begin{de}
A $T$-module $\mathcal{A}=(\mathbb{G}_{a}^{m},\Phi)$ of degree $\widetilde{d}$ and dimension $m$ defined on the field $\mathcal{F}\subset \ov{k}$ is the algebraic group $\mathbb{G}_{a}^{m}$ having the structure of $A$-module given by the $\mathbb {F}_{q}$-algebras homomorphism:\[\Phi:A\to \mathcal{F}^{m,m}\{\tau\}\]\[T\mapsto\sum_{i=0}^{\widetilde{d}}a_{i}(T)\tau^{i};\]where $a_{0}$ (also called $d\Phi(T)$, the \textbf {differential} of $\Phi(T)$, which can be seen as a linear map acting on $\mathcal{C}^{m}$) is of the form:\[a_{0}=TI_{m}+N;\]where $N$ is a nilpotent matrix, and $a_{\widetilde{d}} \neq 0$. This shows moreover that $\Phi$ is injective, as in the case of a Drinfeld module (which is just a $T$-module having dimension $1$).
\end{de}
\begin{de}
A $T-$module of dimension $1$ is called a \textbf{Drinfeld module}. A particularly interesting example of a Drinfeld module is the typical degree $1$ case, called the \textbf{Carlitz module}:\[C=(\mathbb{G}_{a},\Phi);\]where:\[\Phi(T)(\tau):=T+\tau.\]
\end{de}
\begin{de}
The set of torsion points of the $T-$module $\mathcal{A}$ is:\[\mathcal{A}_{tors.}:=\{\ov{x}\in \mathcal{A}, \exists a(T)\in A\setminus\{0\},\Phi(a(T))(\ov{x})=\ov{0}\}.\]
\end{de}
\begin{de}
A sub-$T-$module $\mathcal{B}$ of a $T-$module $\mathcal{A}$ is a reduced connected algebraic subgroup of ($\mathcal{A},+$) such that $\Phi(T)(\mathcal{B})\subset \mathcal{B}$. 
\end{de}
We remark that a sub-$T-$module of a $T-$module is not in general a $T-$module. This can be seen in the following example.
\begin{prop}
Let us consider $q=2$. Let $\mathbb{D}_{1}=(\mathbb{G}_{a},\Phi_{1})$ and $\mathbb{D}_{2}=(\mathbb{G}_{a},\Phi_{2})$ be rank-1 Drinfeld modules such that:\[\Phi_{1}(T)(\tau)=T+T\tau\textit{   and   }\Phi_{2}(T)(\tau)=T+T^{2}\tau.\]Therefore, the algebraic subgroup:\[\mathcal{B}:=\{(x,y)\in \mathbb{G}_{a}^{2},y=x+x^{2}\};\]is a sub-$T-$module of $\mathbb{D}_{1}\times \mathbb{D}_{2}$.
\end{prop}
\begin{proof}
We remark that $\mathcal{B}$ is reduced and connected. Now, for each $(x,y)\in \mathcal{B}$ we have that:\[\Phi_{2}(T)(y)=Ty+T^{2}y^{2}=T(x+x^{2})+T^{2}(x^{2}+x^{4})=\Phi_{1}(T)(x)+\Phi_{1}(T)(x)^{2};\]so that:\[(\Phi_{1}(T)(x),\Phi_{2}(T)(y))\in \mathcal{B}.\]This proves Proposition 1. As $\mathcal{B}$ is not a power of $\mathbb{G}_{a}$ it constitutes an example of a sub-$T-$module which is not a $T-$module.
\end{proof}
\begin{de}
Let $\mathcal{A}=(\mathbb{G}_{a}^{m},\Phi)$ be a $T-$module of dimension $m$. We say that it is \textbf{simple} if it does not admit any non-trivial sub-$T-$module (in other words, a sub-$T-$module different from $\mathcal{A}$ and $0$). 
\end{de}
\begin{de} 
Let $a(T)\in A\setminus \mathbb{F}_{q}$. We call a reduced, connected algebraic sub-group $\mathcal{B}$ of the $T-$module $\mathcal{A}=(\mathbb{G}_{a},\Phi)$ a \textbf{sub-$a(T)-$module} if $\Phi(a(T))(\mathcal{B})\subset \mathcal{B}$. 
\end{de}
We call from now the \textbf{dimension} of a sub-$T-$module $\mathcal{B}$ of $\mathcal{A}$, the dimension of $\mathcal{B}$ as an algebraic variety over $\mathcal{C}$. We remark that the dimension of any non-trivial sub-$T-$module $\mathcal{B}<\mathcal{A}$ is strictly less than the dimension of $\mathcal{A}$.\\\\
Let $\mathcal{B}$ be any non-trivial sub-$T-$module of $\mathcal{A}$. We call for the moment \textbf{torsion set} a subset of $\mathcal{A}$ under the form:\[\ov{x}+\mathcal{B},\texttt{   }\ov{x}\in \mathcal{A}_{tors.}.\]We will extend later such a definition to the more general one of \textbf{torsion subvariety} (see Definition 11) in order to adapt the statement of Theorem 2 and Theorem 4 to the $T-$module context.
\begin{de}
By calling $\mathcal{F}\subset \ov{k}$ the field generated over $k$ by the entries of the coefficient matrices of $\Phi(T)$ (which are in $\ov{k}$), the \textbf{rank} of a $T-$module $\mathcal{A}$ is the rank over $\mathcal{F}[T]$ of the $\mathcal{F}[T]-$module $Hom_{\mathcal{F}}(\mathcal{A},\mathbb{G}_{a})$ (the $\mathbb{F}_{q}-$additive group homomorphisms from $\mathcal{A}$ to $\mathbb{G}_{a}$). We say that a $T-$module $\mathcal{A}$ is \textbf{abelian} if $Hom_{\mathcal{F}}(\mathcal{A},\mathbb{G}_{a})$ has finite rank.
\end{de}
An easy example of an abelian $T-$module is given by a product of Drinfeld modules. It is immediate to see indeed that the rank of a Drinfeld module coincides with its degree. By taking $\mathbb{D}_{1}, ..., \mathbb{D}_{r}$ finitely many Drinfeld modules of rank, respectively, $d_{1}, ..., d_{r}$ and defined over the field $\mathcal{F}\subset \ov{k}$ one sees that:\[Hom_{\mathcal{F}}(\mathbb{D}_{1}\times\cdots \times \mathbb{D}_{r},\mathbb{G}_{a})\simeq \bigoplus_{i=1}^{r}Hom_{\mathcal{F}}(\mathbb{D}_{i},\mathbb{G}_{a});\]so that the rank of such a product of Drinfeld modules is $\sum_{i=1}^{r}d_{i}$. This example shows also that the rank of a $T-$module is not in general the degree of this one. Indeed, the degree of $\mathbb{D}_{1}\times\cdots \times \mathbb{D}_{r}$ is $\max_{i=1, ..., r}\{d_{i}\}$.\\\\
We provide an example of a nonabelian $T-$module in Proposition 7.\\\\
One can prove (cfr. \cite{Goss}, Theorem 5.4.10) that if $\mathcal{A}$ is abelian, $Hom_{\mathcal{F}}(\mathcal{A}, \mathbb{G}_{a})$ is also free as a $\mathcal{F}[T]-$module.\\\\
Let $\mathcal{A}$ be a $T-$module of dimension $m$. We call $Lie(\mathcal{A})\simeq \mathcal{C}^{m}$ the tangent space of $\mathcal{A}$ at $\ov{0}$. 
\begin{de}
Let $\mathcal{A}$ be an abelian $T-$module. The \textbf{exponential function} of $\mathcal{A}$ is the \textbf{unique}\footnote{See \cite{Goss}, Proposition 5.9.2.} morphism:\[\ov{e}:Lie(\mathcal{A})\to \mathcal{A};\]such that, for each $\ov{z}\in Lie(\mathcal{A})$, we have that:\[\ov{e}(d\Phi(T)\ov{z})=\Phi(T)(\ov{e}(\ov{z}));\]as described in \cite{Goss}, Definition 5.9.7. 
\end{de}
It is known (see \cite{Goss}, section 5) that such a morphism is $\mathbb{F}_{q}-$linear, a local homeomorphism and (see Definition 9) $\mathcal{F}-$entire too.\\\\If $\mathcal{A}$ is abelian and we write:\[\Lambda:=Ker(\ov{e});\]this kernel is an $A-$lattice inside $Lie(\mathcal{A})$ and its $A-$rank is less or equal the rank of $\mathcal{A}$, cfr. \cite{Goss}, Lemma 5.9.12.\\\\
The exponential map associated to $\mathcal{A}=(\mathbb{G}_{a}^{m},\Phi)$ projects therefore $\mathcal{C}^{m}$ in $\mathbb{G}_{a}^{m}$. 
\begin{prop}
Let $\mathcal{B}$ be a sub-$a(T)-$module of the abelian $T-$module $\mathcal{A}$ for a given $a(T)\in A\setminus \mathbb{F}_{q}$, and $Lie(\mathcal{B})$ its tangent space, contained as a $\mathcal{C}-$subspace into $Lie(\mathcal{A})$. Therefore:\[\ov{e}(Lie(\mathcal{B}))\subseteq \mathcal{B}.\]
\end{prop}
\begin{proof}
By calling:\[t:=a(T);\]and seeing $\mathcal{B}$ as a sub-$t-$module of the $t-$module $\mathcal{A}$, the statement follows directly from \cite{Goss} Remark 5.9.8.
\end{proof}
As $\ov{e}(Lie(\mathcal{B}))\subseteq \mathcal{B}$ it follows that:\[Lie(\mathcal{B})\subseteq \ov{e}^{-1}(\mathcal{B}).\]

\begin{lem}
Let $\rho(\Lambda_{\mathcal{A}})$ be the $A-$rank of the lattice associated to $\mathcal{A}$ as the kernel of the exponential function $\ov{e}:Lie(\mathcal{A})\to \mathcal{A}$, and let $\rho(\mathcal{A})$ be the rank of $\mathcal{A}$. The following properties are equivalent for an abelian $T-$module $\mathcal{A}=(\mathbb{G}_{a}^{m},\Phi)$.
\begin{enumerate}
	\item $\rho(\Lambda_{\mathcal{A}})=\rho(\mathcal{A})$;
	\item The exponential function $\ov{e}:Lie(\mathcal{A})\to \mathcal{A}$ is surjective.
\end{enumerate}
\end{lem}
\begin{proof}
See \cite{Goss}, Theorem 5.9.14.
\end{proof}
\begin{de}
Let $\mathcal{A}=(\mathbb{G}_{a}^{m},\Phi)$ be an abelian $T-$module. If it respects the two equivalent conditions of Lemma 1 it is called \textbf{uniformizable}.
\end{de}
We remark that the proof of Lemma 1 in \cite{Goss}, Theorem 5.9.14 also applies to sub-$T-$modules of a given abelian $T-$module by the slightly extended notion of $T-$module considered in such a work. We can therefore easily extend Definition 9 to sub-$T-$modules too. More specifically, we say that a sub-$T-$module $\mathcal{B}$ of a given abelian $T-$module $\mathcal{A}$ is \textbf{uniformizable} if the induced exponential map:\[\ov{e}:Lie(\mathcal{B})\to \mathcal{B};\]is surjective.
\begin{rem}
The sub-$T-$modules of an abelian, uniformizable $T-$module are abelian and uniformizable.
\end{rem}
\begin{proof}
See \cite{D}, Remarque 2.1.8 and Remarque 2.1.19.
\end{proof}
We also remark (see \cite{D}, Lemme 2.1.37) that if $\mathcal{A}$ is an abelian, uniformizable $T-$module of rank $d$, then $Lie(\mathcal{A})$ can be written as a direct sum of a $k_{\infty}-$vector space of dimension $d$, into which the associated lattice $\Lambda=Ker(\ov{e}_{\mathcal{A}})$ is cocompact (this is called the \textbf{torsion part} of $Lie(\mathcal{A})$), and a $k_{\infty}-$vector space of infinite dimension (the \textbf{free part} of $Lie(\mathcal{A})$). In other words, by calling $\{\ov{\omega}_{1}, ..., \ov{\omega}_{d}\}$ a generating set of periods of:\[\Lambda=<\ov{\omega}_{1}, ..., \ov{\omega}_{d}>_{A};\]we have that:\[Lie(\mathcal{A})\simeq (\bigoplus_{i=1}^{d}k_{\infty}\ov{\omega}_{i})\oplus Free_{k_{\infty}};\]where $Free_{k_{\infty}}$ is an infinite-dimensional $k_{\infty}-$vector space. Up to the automorphism $\phi$ of $Lie(\mathcal{A})$ which sends $\{\ov{\omega}_{1}, ..., \ov{\omega}_{d}\}$ to the canonical basis of $k_{\infty}^{d}$ and leaves unchanged the free part of $Lie(\mathcal{A})$, so that it moves $\Lambda$ to $A^{d}$, we have then that:\[Lie(\mathcal{A})\simeq (k_{\infty}/A)^{d}\oplus Free_{k_{\infty}};\]so that $\ov{e}$ put in bijection the set $\mathcal{A}_{tors.}$ of the torsion points of $\mathcal{A}$ with the set $(k/A)^{d}\times \ov{0}$.\\\\
By the discussion above it is now easy to see that the torsion points of $\mathcal{B}$ with respect to its structure of sub-$a(T)-$module of $\mathcal{A}$ (in other words, the $a(T)-$torsion points of $\mathcal{A}$ which belong to $\mathcal{B}$) shall correspond via the exponential map $\ov{e}$ associated to $\mathcal{A}$ to the $\mathbb{F}_{q}(a(T))-$rational points of $\phi(Lie(\mathcal{B}))\cap((k_{\infty}/A)^{d}\times \ov{0})$.\\\\
As one can rapidly check a $T-$module which is abelian and uniformizable can be interpreted therefore by analyzing its associated finite-rank lattice in an analogous fashion as for an abelian variety. The new techniques recently introduced by U. Zannier and J. Pila in \cite{PZ} which provided an elegant alternative proof of Theorem 4 involving Diophantine Geometry, are specifically based on such an interpretation and this is the reason why we focus on this particular class of $T-$modules.
\section{A new conjecture}
In this section we will study the possibility to adapt the Manin-Mumford conjecture to the $T-$modules as they have been previously introduced.\\\\
Results about a connection between the Mordell-Lang and the Manin-Mumford conjectures for $T-$modules have been worked out firstly by L. Denis (see \cite{Denis'}), who proposed the following unified conjecture. 
\begin{st}
Let $\mathcal{A}=(\mathbb{G}_{a}^{m},\Phi)$ be a $T-$module of rank $d$ and dimension $m>1$. Let $\Gamma$ be a finitely generated submodule of $\mathcal{A}(\ov{k})$ and let $X$ be a closed subvariety of $\mathbb{G}_{a}^{m}$. Let:\[\ov{\Gamma}:=\{\ov{x}\in \mathcal{A}(\ov{k}),\exists a(T)\in A\setminus\{0\},\Phi(a(T))(\ov{x})\in \Gamma\}.\]Therefore, there exist finitely many translates of sub-$T-$modules of $\mathcal{A}$ in the form $\ov{\gamma}_{1}+\mathcal{B}_{1}, ..., \ov{\gamma}_{s}+\mathcal{B}_{s}$ such that:\[X\cap \ov{\Gamma}=\bigcup_{1\leq i\leq s}(\ov{\gamma}_{i}+\mathcal{B}_{i}\cap \ov{\Gamma}).\]
\end{st}
As we have seen in the previous section, such a statement would imply an intuitive adaptation of Mordell-Lang and Manin-Mumford conjectures to the $T-$modules context. We will see anyway in the present section that Statement 1 is unfortunately false.\\\\
L. Denis proved on the other hand in the same paper, under some technical restriction on the hypotheses, an adapted formulation of the Manin-Mumford conjecture for finite powers of Drinfeld modules (see \cite{Denis'}, Théorème 1). 
Such a result has been subsequently extended into a completely analogous formulation of Theorem 2 for finite powers of Drinfeld modules by T. Scanlon in \cite{Scanlon}, removing so the restriction on the hypotheses in L. Denis' result summarized above. T. Scanlon's result is the following one. 
\begin{thm}
Let $\mathcal{A}=\mathbb{D}^{m}:=(\mathbb{G}_{a}^{m},\Phi_{m})$ be a power of some given Drinfeld module $\mathbb{D}=(\mathbb{G}_{a},\Phi)$ defined over a field $\mathcal{F}$, so that $\Phi_{m}(T)=\Phi(T)I_{m}$ (in other words, $\Phi$ acts diagonally on $\mathbb{G}_{a}^{m}$). Let $X$ be an irreducible algebraic sub-variety of $\mathcal{A}$. If $X(\mathcal{F})_{tors.}$ is Zariski-dense in $X$, then $X$ is the translate of some sub-$T-$module of $\mathcal{A}$ by a torsion point.
\end{thm}
T. Scanlon proved also an adapted version of the Mordell-Lang conjecture to powers of Drinfeld modules of finite characteristic\footnote{See \cite{Goss}, chapter 4.} (see \cite{Scanlon'}), analogous to Theorem 1:
\begin{thm}
Let $\mathbb{D}$ be a Drinfeld module of finite characteristic and modular trascendence degree\footnote{See the referred paper for the complete definition.} at least 1. Let $\Gamma$ be a finitely generated submodule of $\mathbb{D}^{m}(\ov{k})$ for some $m>1$ and let $X$ be a closed $\ov{k}-$subvariety of $\mathbb{G}_{a}^{m}$. Then $X(\ov{k})\cap \Gamma$ is a finite union of translates of subgroups of $\Gamma$.
\end{thm}
As we will highlight better later such a specific case of $T-$module (a finite power of a Drinfeld module) presents particularly strong properties which make it one of the best cases of study. A $T-$module in such a form is moreover easily abelian and uniformizable 
and this would appear to be a first encouraging step for our project to prove a similar result, using the new techniques introduced in \cite{PZ} for an abelian and uniformizable $T-$module $\mathcal{A}$. We would like therefore to prove a weaker result, analogous to Theorem 4, but for a general abelian and uniformizable $T-$module, as in the following formulation. 
\begin{for}
Let $X$ be an algebraic subvariety of an abelian uniformizable $T-$module $\mathcal{A}$ defined over $\ov{k}$. If $X$ does not contain any torsion set of $\mathcal{A}$, then $X$ contains at most finitely many torsion points.
\end{for}
Such a statement turns out to be unfortunately false. This is an immediate consequence of the following proposition.
\begin{prop}
There exists an abelian uniformizable $T-$module $\mathcal{A}$ and a non-trivial algebraic subvariety of $\mathcal{A}$ which contains infinitely many torsion points but no torsion sets of $\mathcal{A}$.
\end{prop}
\begin{proof}

We consider the $T-$module of dimension $2$ defined by the tensor power $C^{\otimes 2}=(\mathbb{G}_{a}^{2},\Phi)$ of the Carlitz module $C$, introduced by G. Anderson and D. Thakur in \cite{An-Thk}. We suppose that $q=2$. Such a $T-$module is then under the following form:\[\Phi(T)(\tau)=\left(\begin{array}{cc}T&1\\0&T\end{array}\right)+\left(\begin{array}{cc}0&0\\1&0\end{array}\right)\tau.\]We can then show that for each $\left(\begin{array}{c}X\\Y\end{array}\right)\in C^{\otimes 2}$:\[\Phi(T^{2})\left(\begin{array}{c}X\\Y\end{array}\right)=\left(\begin{array}{c}T^{2}X+X^{2}\\T^{2}Y+(T+T^{2})X^{2}+Y^{2}\end{array}\right).\]The algebraic sub-group $0\times \mathbb{G}_{a}$ of $C^{\otimes 2}$ is then easily a sub-$T^{2}-$module, but not a sub-$T-$module
. Indeed, any tensor power $C^{\otimes m}$, for any power $m\in \mathbb{N}\setminus\{0\}$, of the Carlitz module $C$ is always a simple, abelian (see \cite{Goss}, Corollary 5.9.38) and uniformizable $T-$module (see \cite{Yu}, Proposition 1.2), but one can prove that it possesses sometimes (as in the present case) non-trivial sub-$T^{j}-$modules for some $j$ depending on $m$ and $q$. By choosing $0\times \mathbb{G}_{a}$ as an algebraic subvariety of $C^{\otimes 2}$, we now see that it contains infinitely many torsion points, which correspond, by Proposition 2 and discussion subsequent to Remark 1, to the $\mathbb{F}_{2}(T^{2})-$rational points of the torsion part of:\[Lie(0\times \mathbb{G}_{a})/(Ker(\ov{e}_{C^{\otimes 2}})\cap Lie(0\times \mathbb{G}_{a})).\]
As $C^{\otimes 2}$ is simple as a $T-$module, $0\times \mathbb{G}_{a}$ cannot contain on the other hand any torsion set of positive dimension, which finally prove our statement. 
\end{proof}
Proposition 3 put in light how the structure of the generic polynomial ring $\mathbb{F}_{q}[T]$, which contains infinitely many subrings isomorphic to $\mathbb{F}_{q}[T]$ itself, determines the failure of any attempt to adapt the Manin-Mumford conjecture (even in its weaker formulation stated in Theorem 4) to abelian uniformizable $T-$modules, while the same conjecture is on the contrary true for abelian varieties over number fields because their algebraic structure is that of a group, which is a $\mathbb{Z}-$module and $\mathbb{Z}$ do not contain as a ring non-trivial subrings.\\\\
A similar phenomenon has been observed already studying Mordell-Lang conjecture for powers of Drinfeld modules of finite characteristic. More precisely, given a finitely generated submodule $\Gamma$ of $\mathbb{D}^{m}$ for some Drinfeld module $\mathbb{D}=(\mathbb{G}_{a},\Phi)$ and $m>1$, if $X$ is an algebraic subvariety of $\mathbb{D}^{m}$ then $X(\ov{k})\cap \Gamma$ might not be stabilized by the action of $\Phi(T)$ but it can be invariant anyway under the action of $\Phi(T^{n})$ for a suitable $n>1$. A detailed example can be found in \cite{Ghioca}, Remark 4.11. This leads to a formulation of Mordell-Lang conjecture for powers of some convenient cases of Drinfeld modules which extends in some sense the notion of submodule of $\mathbb{D}^{m}$ to that of \textit{sub-$\Phi(T^{n})-$module} for $n\in \mathbb{N}\setminus\{0\}$ depending on the chosen $\Phi$, $\Gamma$ and $X$ (see \cite{Ghioca}, Theorem 4.6).\\\\
We extend therefore the class of algebraic sub-modules of $\mathcal{A}$ in order to avoid counter-examples produced, as we showed in the proof of Proposition 3, by the abundance of subrings of $\mathbb{F}_{q}[T]$.\\\\
We start by giving the following definition.
\begin{de}
Let $\mathcal{A}=(\mathbb{G}_{a}^{m},\Phi)$ be a general $T-$module. Let $B$ be a sub-ring of $A$. We call \textbf{sub-$B-$module} of $\mathcal{A}$ any reduced connected algebraic sub-group $\mathcal{B}$ of $\mathbb{G}_{a}^{m}$ such that:\[\Phi(a(T))(\mathcal{B})\subseteq \mathcal{B},\texttt{   }\forall a(T)\in B.\]We say that a sub-set of $\mathcal{A}$ is a \textbf{torsion subvariety} if it is under the following form:\[\ov{x}+\mathcal{B};\]where $\ov{x}\in \mathcal{A}_{tors.}$ and there exists $B$ a sub-ring of $A$ such that $\mathcal{B}$ is a sub-$B-$module of $\mathcal{A}$.
\end{de}
The following result will allow us to ease the study of the sub-$B-$modules of $\mathcal{A}$ for each $B$ subring of $A$, reducing all of them to sub-$T^{j}-$modules for a convenient index $j\in \mathbb{N}\setminus\{0\}$ only depending on $\mathcal{A}$. 
\begin{thm}
Let $\mathcal{A}=(\mathbb{G}_{a}^{m},\Phi)$ be an abelian uniformizable $T-$module of dimension $m$. There exists then a number $j(\mathcal{A})\in \mathbb{N}\setminus\{0\}$ only depending on $\mathcal{A}$ such that for each 
$B$ sub-ring of $A$, every sub-$B-$module of $\mathcal{A}$ is a sub-$T^{j(\mathcal{A})}-$module). 
\end{thm}
\begin{proof}
Let $N$ be the nilpotent matrix associated to the differential $d\Phi(T)$ of $\mathcal{A}$ introduced in Definition 1. Let $n(\mathcal{A})\in \mathbb{N}\setminus\{0\}$ be its order. Let $\mathcal{B}$ be a sub-$B-$module of $\mathcal{A}$ for some $B$ subring of $A$. It is then a sub-$a(T)-$module for each $a(T)\in B\setminus \mathbb{F}_{q}$ and in particular a reduced connected algebraic sub-group of $\mathbb{G}_{a}^{m}$. In order to prove the statement it will be sufficient to show that $Lie(\mathcal{B})$ is stabilized by the action of $d\Phi(T^{j(\mathcal{A})})$ for some convenient $j(\mathcal{A})\in \mathbb{N}\setminus\{0\}$. The reason of this easily comes from Proposition 2. We choose therefore:\[j(\mathcal{A})=p^{r(\mathcal{A})};\]by calling $p^{r(\mathcal{A})}$ the smallest power of $p$ to be greater or equal to $n(\mathcal{A})$ (in other words, $r(\mathcal{A})=[\log_{p}(n(\mathcal{A}))]+1$, where we mean by $[\log_{p}(n(\mathcal{A}))]$ the integer part of $\log_{p}(n(\mathcal{A}))$). It is easy to see then that:\[d\Phi(T^{j(\mathcal{A})})=T^{j(\mathcal{A})}I_{m};\]which stabilizes every vector sub-space of $Lie(\mathcal{A})$ over $\mathcal{C}$. So in particular, it stabilizes $Lie(\mathcal{B})$ too.
\end{proof}
We remark that we just proved the existence of such a number $j(\mathcal{A})\in \mathbb{N}\setminus\{0\}$, but we did not found actually its minimal possible value in principle. We can anyway remark that such a value is $1$ in the case where $\mathcal{A}$ is a power of a Drinfeld module, as we will show here below. 
This shows that for each subring $B$ of $A$  every sub-$B-$module of $\mathcal{A}$ is actually a sub-$T-$module. This is precisely the reason why examples like the one presented in the proof of Proposition 3 would not work for some finite power of a Drinfeld module, being so another confirmation of Scanlon's result (Theorem 5).\\\\
We show now this property using the following Thiery's Theorem (see \cite{T}).
\begin{thm}
Let $\mathbb{D}^{m}=(\mathbb{G}_{a}^{m},\Phi^{m})$ the $m-$th power of a Drinfeld module $\mathbb{D}=(\mathbb{G}_{a},\Phi)$ and let $\mathcal{F}$ be its coefficients field. There exists therefore a bijective correspondence between the family of all sub-$T-$modules of $\mathbb{D}^{m}$ and the family of the vector sub-spaces of $Lie(\mathbb{D}^{m})$ which are $End_{\mathcal{F}}(\Phi)-$rational. This correspondence is given by the exponential function:\[\ov{e}:Lie(\mathbb{D}^{m})\to \mathbb{D}^{m};\]\[V\mapsto \ov{e}(V).\]Moreover, we have that the dimension of any sub-$T-$module of $\mathbb{D}^{m}$ and of the $End_{\mathcal{F}}(\Phi)-$rational vector sub-space of $Lie(\mathbb{D}^{m})$ corresponding to it are the same.
\end{thm}
\begin{proof}
See \cite{T}, Theorem at page 33.
\end{proof}
\begin{prop}
Let $\mathbb{D}^{m}$ be as in Theorem 5. Then:\[j(\mathbb{D}^{m})=1.\]
\end{prop}
\begin{proof}
Let $\mathcal{B}$ be a sub-$B-$module of $\mathbb{D}^{m}$ for a non-trivial subring $B$ of $A$. It is therefore a sub-$a(T)-$module of $\mathcal{A}$ for any $a(T)\in B\setminus \mathbb{F}_{q}$. We assume that the Drinfeld module $\mathbb{D}$ has rank $d$. We define:\[T':=a(T);\]and:\[\Psi(T'):=\Phi(a(T));\]$\mathcal{B}$ is therefore a $T'-$module which is the $m-$th power of some $\mathbb{F}_{q}[T']-$Drinfeld module
. By Theorem 8 $\mathcal{B}$ is the zero locus of $s=m-\dim(\mathcal{B})$ linear equations under the following form:\[\sum_{j=1}^{m}P_{ij}(\tau)X_{j}=0;\]where $P_{ij}(\tau)\in End_{\mathcal{F}}(\Psi)$ for each $i=1, ..., s$ and each $j=1, ..., m$. Now, it is well known that $\Phi(T)\in End_{\mathcal{F}}(\Psi)$. 
Moreover, $End_{\mathcal{F}}(\Psi)$ is a commutative ring because $\mathbb{D}$ has characteristic $0$ (see \cite{Goss}, Definition 4.4.1). Therefore, one sees that $\mathcal{B}$ is actually a sub-$T-$module of $\mathbb{D}^{m}$ as it is stabilized by the action of $\Phi(T)$.
\end{proof}
If $\mathcal{A}$ is \textbf{absolutely simple} (this means that it does not contain any non-trivial sub-$T^{j}-$module for each $j\in \mathbb{N}\setminus\{0\}$) we fix $j(\mathcal{A})=1$.\\\\
We also remark that this study of powers of some Drinfeld module provides an example of a class of $T-$modules $\mathcal{A}$ such that for each $i\in \mathbb{N}\setminus\{0\}$ we have that $j(\mathcal{A}^{i})=j(\mathcal{A})$ and one may wonder if it is a general phenomenon.\\\\
This would lead us to a more general formulation of Manin-Mumford conjecture on abelian and uniformizable $T-$modules, 
where we propose to show that, up to a finite number, the torsion points of $\mathcal{A}$ could be shared in finitely many sub-$T^{j(\mathcal{A})}-$modules. We would like then to state an analogue of Theorem 4 as in the following formulation.
\begin{for}
Let $X$ be an algebraic subvariety of an abelian uniformizable $T-$module $\mathcal{A}$. If $X$ does not contain any torsion subvariety of $\mathcal{A}$ of dimension $>0$, then $X$ contains at most finitely many torsion points.
\end{for}
Such a new formulation is however false again and not yet sufficient to completely exclude other counter-examples. The nice one which follows has been suggested by Laurent Denis. 
\begin{prop}
There exists an algebraic subvariety $X$ of an abelian uniformizable $T-$module $\mathcal{A}$ which does not contain non-trivial torsion subvarieties of $\mathcal{A}$ but it contains infinitely many torsion points.
\end{prop}
\begin{proof}
We consider the following case of a product of two non-isogeneous Drinfeld modules. We assume that $q=2$. Let $\mathbb{D}_{1}=(\mathbb{G}_{a},C)$ be the Carlitz module, so that $C(T)(\tau)=T+\tau$, being $\tau$ the Frobenius automorphism over $\ov{\mathbb{F}_{2}}$. We define the Drinfeld module $\mathbb{D}_{2}=(\mathbb{G}_{a},C_{(2)})$ as follows:\[C_{(2)}(T)(\tau):=T+(T^{1/2}+T)\tau+\tau^{2}.\]One can see then that for each $z\in \mathcal{C}$ we have that:\[C_{(2)}(T)(\sqrt{z})=\sqrt{C(T^{2})(z)}.\]The product $\mathbb{D}_{1}\times \mathbb{D}_{2}=(\mathbb{G}_{a}^{2},\Phi)$ such that:\[\Phi(T)\left(\begin{array}{c}X\\Y\end{array}\right):=\left(\begin{array}{c}C(T)(X)\\C_{(2)}(T)(Y)\end{array}\right);\]is then a $T-$module as in Definition 1 and for each $z\in C_{tors.}$ we have that $(z,z^{1/2})\in (\mathbb{D}_{1}\times \mathbb{D}_{2})_{tors.}$. The algebraic variety $X=Y^{2}$ contains then all these infinitely many torsion points and, as $C(T^{j})(Y^{2})\neq (C_{(2)}(T^{j})(Y))^{2}$ for each $j\in\mathbb{N}\setminus\{0\}$, it is not stabilized by the action of $\Phi(T^{j})$. It can not be then a sub-$T^{j}-$module of $\mathbb{D}_{1}\times \mathbb{D}_{2}$ for any $j\in \mathbb{N}\setminus\{0\}$. As this variety has $\mathcal{C}-$dimension $1$, it can not admit non-trivial sub-$T^{j}-$modules either.
\end{proof}
We remark that the same counter-example may be repeated identically for any $q$, replacing the Carlitz module $C$ by a generic Drinfeld module $\mathbb{D}_{1}=(\mathbb{G}_{a},\Phi_{1})$ and $C_{(2)}$ by the Drinfeld module $\mathbb{D}_{2}=(\mathbb{G}_{a},\Phi_{2})$ obtained as the $1/q^{s}-$th root of the coefficients of $\Phi_{1}(T^{q^{s}})$, which would define an infinite class of bad cases
.\\\\
We would like to stress also the fact that the example we provided in the above proof of Proposition 5 just involves a product of Drinfeld modules. T. Scanlon's proof of an analogue of Manin-Mumford conjecture for finite powers of a given Drinfeld module, given in \cite{Scanlon}, cannot be repeated in general for a product of different Drinfeld modules. We believe anyway that such an argument may still hold for a product of different Drinfeld modules such that these ones have the same degree.\\\\
The study of the arithmetic of $T-$modules turns out to be therefore much more delicate than one could expect. Even taking count of the existence of non-trivial subrings of $A$ by extending the notion of torsion subvariety to all the possible sub-$B-$modules for every $B$ non-trivial subring of $A$, the structure of the $T-$modules still allows the construction of counter-examples to an analogue of a (even weaker) formulation of the Manin-Mumford conjecture as stated in Theorem 4 for abelian varieties defined over a number field. We propose then some new restrictions to our hypothesis in order to avoid counter-examples like the one described in Proposition 5.\\\\
We consider again a $T-$module which is a product of $m$ Drinfeld modules, non-isogeneous to each other. So we take $\mathcal{A}:=\mathbb{D}_{1}\times...\times \mathbb{D}_{m}=(\mathbb{G}_{a}^{m},\Phi)$. By Definition 1, we remark that if the coefficient matrix $a_{d}$ is invertible, all these Drinfeld modules necessarily have the same rank $d$. As one can easily see, this makes impossible to produce situations as the one showed in the proof of Proposition 5, for any possible choice of $\Phi$ and $q$.\\\\ 
Finally, when we consider a $T-$module whose associated leading matrix is invertible, 
nothing seems to be an obstacle to a result of Manin-Mumford type. 
As such a direct hypothesis 
would exclude a lot of good cases, as for example a tensor power of the Carlitz module, we ask more generally that, given an abelian uniformizable $T-$module $\mathcal{A}=(\mathbb{G}_{a}^{m},\Phi)$, \textbf{there exists a number $i\in \mathbb{N}\setminus\{0\}$ such that the leading coefficient matrix $a'_{id}$ of the $i-$th iterated $\Phi(T^{i})(\tau)$ of $\Phi(T)(\tau)$ is invertible}.\\\\
As it is easy to see, such a more general hypothesis on the $T-$module $\mathcal{A}$ does not change in any significant way our previous argument and still implies that if $\mathcal{A}$ is a product of Drinfeld modules, they have all to have the same rank. 
On the other hand, we see for example that a tensor power of the Carlitz module respects such an hypothesis on the leading coefficient matrix. We can also remark that this hypothesis does actually imply already the condition on the $T-$module to be abelian, as we show in the following Theorem.
\begin{thm}
Let $\mathcal{A}=(\mathbb{G}_{a}^{m},\Phi)$ be a $T-$module defined over the field $\mathcal{F}\subset \ov{k}$, such that there exists a number $i\in \mathbb{N}\setminus\{0\}$ such that the leading coefficient of $\Phi(T^{i})\in \mathcal{F}^{m,m}\{\tau\}$ is an invertible matrix. Therefore, $\mathcal{A}$ is abelian. 
\end{thm}
\begin{proof}
By the canonical isomorphism:\[Hom_{\mathcal{F}}(\mathcal{A},\mathbb{G}_{a})\simeq \mathcal{F}\{\tau\}^{m};\]we consider a morphism $f\in Hom_{\mathcal{F}}(\mathcal{A},\mathbb{G}_{a})$ as an element $f(\tau)\in \mathcal{F}\{\tau\}^{m}$. As the leading coefficient of $\Phi(T^{i})$ is invertible and the Ore algebra $\mathcal{F}\{\tau\}$ is a (non commutative) ring endowed of the right division algorithm (see \cite{Goss}, Proposition 1.6.2), it is possible to divide on right by $\Phi(T^{i})$ each element of $\mathcal{F}\{\tau\}^{m}$, knowing that the coefficients of such an object are vectors in $\mathcal{F}^{m}$ while those of $\Phi(T^{i})$ are matrices in $\mathcal{F}^{m,m}$. In fact, an invertible matrix in $\mathcal{F}^{m,m}$ also divide (on right) each element of $\mathcal{F}^{m}$, and this fact allows the euclidean division. The algebra $\mathcal{F}\{\tau\}^{m}$ could then be shared in $m\widetilde{d}(i)$ (where $\widetilde{d}(i)$ is the degree of $\Phi(T^{i})$ as an additive polynomial in $\tau$) division classes modulo $\Phi(T^{i})$, which is equivalent to say that $Hom_{\mathcal{F}}(\mathcal{A},\mathbb{G}_{a})\simeq \mathcal{F}\{\tau\}^{m}$ is generated by $m\widetilde{d}(i)$ elements as a $\mathcal{F}[T]-$module.
\end{proof}
We remark that the proof of the above Theorem 9 may be used as well to compute the rank of an abelian uniformizable $T-$module which respects the condition on its leading coefficient matrix that we discussed above. We show this in the following proposition.
\begin{prop}
Given the $T-$module $C^{\otimes 2}=(\mathbb{G}_{a}^{2},\Phi)$ as in Proposition 3, so that $q=2$, the $T^{2}-$module $(\mathbb{G}_{a}^{2},\Phi^{2})$ has rank $2$ and its sub-$T^{2}-$module $0\times \mathbb{G}_{a}$ has rank 1.
\end{prop}
\begin{proof}
The tensor square $C^{\otimes 2}$ of a Carlitz module presented in the proof of Proposition 3 does not respect the condition on its leading coefficient matrix to be invertible, but we remarked in such a proof that it can be seen as well as a $t-$module, by calling $t:=T^{2}$. Now, the new $t-$module $(\mathbb{G}_{a}^{2},\Psi)$, where 
$\Psi(t)=\Phi(T^{2})$, so that:\[\Psi(t)(\tau)=\left(\begin{array}{cc}t&0\\0&t\end{array}\right)+\left(\begin{array}{cc}1&0\\(\sqrt{t}+t)&1\end{array}\right)\tau;\]respects all the required hypotheses being its leading matrix invertible. We see therefore by the proof of Theorem 9 that it has rank $2$ being its dimension $2$ and its degree $1$. Such a degree remains the same for $0\times \mathbb{G}_{a}$, which is a sub-$t-$module of dimension 1 of $(\mathbb{G}_{a}^{2},\Psi)$. Its rank will be therefore 1. 
\end{proof}
The methods we used in the proof of Theorem 9 also suggest an entire class of examples of nonabelian $T-$modules. We present a typical case of such $T-$modules in the following proposition. 
\begin{prop}
Let $q=2$. Let $\mathcal{A}=(\mathbb{G}_{a}^{2},\Phi)$ the $T-$module defined by the following action of $\Phi$:\[\Phi(T)(\tau)=TI_{2}+\left(\begin{array}{cc}0&0\\1&0\end{array}\right)\tau.\]Therefore, $\mathcal{A}$ is nonabelian.
\end{prop}
\begin{proof}
We see that:\[\Phi(T^{2})(\tau)=T^{2}I_{2}+\left(\begin{array}{cc}0&0\\T+T^{2}&0\end{array}\right)\tau.\]More specifically, the form of the leading coefficient matrix (which is nilpotent of order 2) is such that the degree in $\tau$ of $\Phi(T^{j})(\tau)$ is 1 for every $j\in \mathbb{N}\setminus\{0\}$. 
Indeed, if one carries on taking iterates of $\Phi(T)(\tau)$ it is easy to see that for every $j\in \mathbb{N}\setminus\{0\}$ they will take the following shape:\[\Phi(T^{j})(\tau)=T^{j}I_{2}+\left(\begin{array}{cc}0&0\\b_{j}(T)&0\end{array}\right)\tau;\]for some $b_{j}(T)\in A$. In other words, for each $j\in \mathbb{N}\setminus\{0\}$ the additive form $\Phi(T^{j})(\tau)$ will always have degree $1$ in $\tau$. As we know that:\[Hom_{k}(\mathcal{A},\mathbb{G}_{a})\simeq k\{\tau\}^{2};\]if $\mathcal{A}$ was abelian there would be a finite set of generating elements $\ov{f}_{1}(\tau), ..., \ov{f}_{r}(\tau)$ of $k\{\tau\}^{2}$ such that for each $\ov{f}(\tau)\in k\{\tau\}^{2}$, there would exist $a_{1}(T), ..., a_{2}(T)\in k[T]$ such that:\[\ov{f}(\tau)=a_{1}(T)\cdot\ov{f}_{1}(\tau)+\cdots +a_{r}(T)\cdot\ov{f}_{r}(\tau);\]where the action:\[h(T)\cdot \ov{g}(\tau);\]for $h(T)=h_{0}+h_{1}T+...+h_{s}T^{s}\in k[T]$ and $\ov{g}(\tau)\in k\{\tau\}^{2}$ comes from the following one:\[h_{i}T^{i}\cdot \ov{g}(\tau):=h_{i}\ov{g}(\tau)\circ \Phi(T^{i});\]as described in \cite{Goss}, page 146. By 
fixing the notation, for $i=1, ..., r$:\[\ov{f}_{i}(\tau)=\left(\begin{array}{c}f_{i,X}(\tau)\\f_{i,Y}(\tau)\end{array}\right);\]one sees that, for each $j\in \mathbb{N}\setminus\{0\}$:
\[\ov{f}_{i}(\tau)\circ \Phi(T^{j})=\left(\begin{array}{c}T^{j}f_{i,X}(\tau)+b_{j}(T)f_{i,Y}(\tau)\tau\\T^{j}f_{i,Y}(\tau)\end{array}\right).\]Therefore, as the multiplication by the coefficients of $a_{1}(T), ..., a_{r}(T)$ (which are in $k$) does not change the degree in $\tau$, it is now possible to see that the degree in $\tau$ of the expression $a_{1}(T)\cdot\ov{f}_{1}(\tau)+\cdots +a_{r}(T)\cdot\ov{f}_{r}(\tau)$ cannot exceed the value:\[\max_{i=1, ..., r}\{\deg_{\tau}(f_{i,X}(\tau)),\deg_{\tau}(f_{i,Y}(\tau))+1\}.\]This clearly contradicts the arbitrary choice of $\ov{f}(\tau)$ and it is a direct consequence of the bounded degree in $\tau$ of all the possible iterates of $\Phi(T)(\tau)$. 
Therefore, $\mathcal{A}$ cannot be abelian as a $T-$module.
\end{proof}
The example we gave in Proposition 7 might be extended to the general class of all $T-$modules $\mathcal{A}=(\mathbb{G}_{a}^{m},\Phi)$ such that there exists a positive integer $N_{\mathcal{A}}>0$ such that for each $i\in \mathbb{N}\setminus\{0\}$ one has that the degree in $\tau$ of $\Phi(T^{i})(\tau)$ is at most $N_{\mathcal{A}}$. This implies the following corollary.
\begin{cor}
A $T-$module $\mathcal{A}=(\mathbb{G}_{a}^{m},\Phi)$ such that there exists $M_{\mathcal{A}}\in \mathbb{N}\setminus\{0\}$ such that for each $i\in \mathbb{N}\setminus\{0\}$ the leading coefficient matrix of the additive form $\Phi(T^{i})(\tau)$ is nilpotent of order $\leq M_{\mathcal{A}}$, is nonabelian.
\end{cor}
We state then a $T-$modules version of Manin-Mumford conjecture in the spirit of U. Zannier's and J. Pila's weak formulation as follows.
\begin{conj}
Let $\mathcal{A}=(\mathbb{G}_{a}^{m},\Phi)$ be a uniformizable $T-$module of dimension $m>1$ such that there exists $i\in \mathbb{N}\setminus\{0\}$ such that the leading coefficient matrix of $\Phi(T^{i})$ is invertible. 
Let $X$ be a non-trivial irreducible algebraic sub-variety of $\mathcal{A}$, defined over $\ov{k}$. If $X$ does not contain torsion subvarieties of dimension $>0$, $X$ contains at most finitely many torsion points of $\mathcal{A}$.
\end{conj}
Our strategy to prove Conjecture 1 is based on the techniques developped by U. Zannier and J. Pila in \cite{PZ}. We refer to \cite{D}, chapter 2, for a complete discussion. Roughly speaking, we essentially work on the decomposition briefly described after Remark 1 of the tangent space $Lie(\mathcal{A})$ of an abelian uniformizable $T-$module $\mathcal{A}$ in its torsion part and in its free part. This allows to translate the computation of the torsion points of $\mathcal{A}$ contained in some algebraic subvariety $X$ of $\mathcal{A}$ into a Diophantine Geometry problem of finding out the $k-$rational points of the torsion part of the analytic set $Y:=\ov{e}^{-1}(X)$ contained in $Lie(\mathcal{A})$.\\\\
Always considering such a particular class of $T-$modules which present so many relevant analogies with classical abelian varieties we might propose the analogue of the stronger version of Manin-Mumford conjecture as follows. 
\begin{conj}
Let $\mathcal{A}=(\mathbb{G}_{a}^{m},\Phi)$ be a uniformizable $T-$module of dimension $m>1$ such that there exists $i\in \mathbb{N}\setminus\{0\}$ such that the leading coefficient matrix of $\Phi(T^{i})$ is invertible. Let $X$ be a non-trivial irreducible algebraic subvariety of $\mathcal{A}$, defined over $\ov{k}$. If $X$ contains a Zariski-dense set of torsion points of $\mathcal{A}$, then $X$ is a torsion subvariety of $\mathcal{A}$.
\end{conj}
We conclude by remarking that our examples show not only that our Formulation 1 and Formulation 2 are false, but also that L. Denis' Statement 1 cannot be true either, as it would obviously imply Formulation 1 and Formulation 2. We would like therefore to propose a modification of Statement 1 into the following unified conjecture. 
\begin{conj}
Let $\mathcal{A}=(\mathbb{G}_{a}^{m},\Phi)$ be a uniformizable $T-$module of dimension $m$ such that there exists $i\in \mathbb{N}\setminus\{0\}$ such that the leading coefficient matrix of $\Phi(T^{i})$ is invertible. Let $X$ be an algebraic subvariety of $\mathbb{G}_{a}^{m}$ , $\Gamma$ be a finitely generated subgroup of $\mathcal{A}(\ov{k})$ and $\ov{\Gamma}$ be like in Statement 1. Therefore, there exist finitely many torsion subvarieties $\ov{\gamma}_{1}+\mathcal{B}_{1}, ..., \ov{\gamma}_{s}+\mathcal{B}_{s}$ of $\mathcal{A}$ such that:\[X\cap \ov{\Gamma}=\bigcup_{1\leq i\leq s}(\ov{\gamma}_{i}+\mathcal{B}_{i}\cap \ov{\Gamma}).\]
\end{conj}
Our Conjecture 3 would imply in particular (again by Proposition 4) D. Ghioca's one (see \cite{Ghioca}) for powers of Drinfeld modules (proved in the finite characteristic case by T. Scanlon (see \cite{Scanlon'}) and under different technical restriction by D. Ghioca in the same paper):
\begin{conj}
Let $\mathbb{D}$ be a Drinfeld module. If $\Gamma$ is a finitely generated submodule of $\mathbb{G}_{a}^{m}(\ov{k})$ for some $m>0$ and if $X$ is an algebraic subvariety of $\mathbb{G}_{a}^{m}$, then there are finitely many sub-$T-$modules $\mathcal{B}_{1}, ..., \mathcal{B}_{s}$ of $\mathbb{D}^{m}$ and elements $\ov{\gamma}_{1}, ..., \ov{\gamma}_{s}\in \mathbb{G}_{a}^{m}(\ov{k})$ such that:\[X(\ov{k})\cap \Gamma=\bigcup_{1\leq i\leq s}(\ov{\gamma}_{i}+\mathcal{B}_{i}(\ov{k})\cap \Gamma).\]
\end{conj}

\end{document}